\documentclass[11pt]{amsart}

\theoremstyle{plain}
\newtheorem{theorem}                {Theorem}      [section]
\newtheorem*{theorem1}                {Theorem \ref{thm:spheres}}
\newtheorem*{theorem2}                {Theorem \ref{th:appl3}}
\newtheorem{proposition}  [theorem]  {Proposition}
\newtheorem{corollary}    [theorem]  {Corollary}
\newtheorem{lemma}        [theorem]  {Lemma}

\theoremstyle{definition}

\newtheorem{definition}   [theorem]  {Definition}

\DeclareMathOperator{\trace}{trace}

 \DeclareMathOperator{\id}{I}

\DeclareMathOperator{\cst}{constant}

\numberwithin{equation}{section}

\begin{document}

\title[Surfaces with parallel mean curvature]
{Surfaces with parallel mean curvature in $\mathbb{S}^3\times\mathbb{R}$ and $\mathbb{H}^3\times\mathbb{R}$}

\author{Dorel~Fetcu}
\author{Harold~Rosenberg}

\address{Department of Mathematics\\
"Gh. Asachi" Technical University of Iasi\\
Bd. Carol I no. 11 \\
700506 Iasi, Romania} \email{dfetcu@math.tuiasi.ro}

\curraddr{IMPA\\ Estrada Dona Castorina\\ 110, 22460-320 Rio de
Janeiro, Brasil} \email{dorel@impa.br}

\address{IMPA\\ Estrada Dona Castorina\\ 110, 22460-320 Rio de
Janeiro, Brasil} \email{rosen@impa.ro}

\thanks{The first author was supported by a Post-Doctoral Fellowship "P\'os-Doutorado S\^enior
(PDS)" offered by FAPERJ, Brazil.}

\begin{abstract} We prove a Simons type equation for non-minimal surfaces with parallel mean curvature vector (pmc surfaces) in $M^3(c)\times\mathbb{R}$, where $M^3(c)$ is a $3$-dimensional space form.
Then, we use this equation in order to characterize certain complete non-minimal pmc surfaces.
\end{abstract}



\maketitle

\section{Introduction}

In 1968, J. Simons discovered a fundamental formula for the
Laplacian of the second fundamental form of a minimal submanifold in
a Riemannian manifold.  He then used it  to characterize certain
minimal submanifolds of a sphere and  Euclidean space (see
\cite{JS}). One year later, K. Nomizu and B. Smyth generalized
Simons' equation for hypersurfaces of constant mean curvature (cmc
hypersurfaces) in a space form (see \cite{NS}). This was extended,
in B. Smyth's work \cite{S}, to the more general case of a
submanifold with parallel mean curvature vector (pmc submanifold) in
a space form. Over the years such equations, called Simons type
equations, turned out to be very useful, a great number of authors
using them in the study of cmc and pmc submanifolds (see, for
example, \cite{AdC,E,WS}).

Nowadays, the study of cmc surfaces in the Euclidean space and, more
generally, in space forms is a classical subject in the field of
Differential Geometry, well known papers like \cite{HH} by H. Hopf
and \cite{C} by S.-S. Chern being representative examples for the
literature on this topic. When the codimension is greater than $1$,
a natural generalization of cmc surfaces are pmc surfaces, which
have been intensively studied in the last four decades, among the
first papers devoted to this subject being \cite{DF} by D. Ferus,
\cite{CL} by B.-Y. Chen and G. D. Ludden, \cite {DH} by D. A.
Hoffman and \cite{Y} by S.-T. Yau. All results in these papers were
obtained in the case when the ambient space is a space form.

The next natural step was taken by U. Abresch and H. Rosenberg, who
studied in \cite{AR,AR2} cmc surfaces and obtained Hopf type results
in product spaces  $M^2(c)\times\mathbb{R}$, where $M^2(c)$ is a
complete simply-connected surface with constant curvature $c$, as
well as the homogeneous $3$-manifolds $Nil(3)$,
$\widetilde{PSL(2,\mathbb{R})}$ and Berger spheres. Some of their
results in \cite{AR} were extended to pmc surfaces in product spaces
of  $M^n(c)\times\mathbb{R}$, where $M^n(c)$ is an $n$-dimensional
space form, in \cite{AdCT1,AdCT} by H. Alencar, M. do Carmo and R.
Tribuzy.

In a  recent paper M. Batista  derived a Simons type equation,
involving the traceless part of the second fundamental form of a cmc
surface in $M^2(c)\times\mathbb{R}$ (see \cite{B}), and found
several applications.

In this paper we compute the Laplacian of the squared norm of the
traceless part $\phi$ of the second fundamental form $\sigma$ of a pmc
surface in a product space $M^3(c)\times\mathbb{R}$ and, using this
Simons type formula, we characterize some of these surfaces. The main results of our paper are the following.

\begin{theorem1} Let $\Sigma^2$ be an immersed pmc $2$-sphere in $M^n(c)\times\mathbb{R}$, such that
\begin{enumerate}
\item $|T|^2=0$ or $|T|^2\geq\frac{2}{3}$ and $|\sigma|^2\leq c(2-3|T|^2)$, if $c<0$;

\item $|T|^2\leq\frac{2}{3}$ and $|\sigma|^2\leq c(2-3|T|^2)$, if $c>0$,
\end{enumerate}
where $T$ is the tangent part of the unit vector $\xi$ tangent to
$\mathbb{R}$. Then, $\Sigma^2$ is either a minimal surface in a totally umbilical hypersurface of $M^n(c)$ or a standard sphere in $M^3(c)$.
\end{theorem1}

\begin{theorem2} Let $\Sigma^2$ be an immersed
complete non-minimal pmc surface in $\bar M=M^3(c)\times\mathbb{R}$,
with $c>0$ and mean curvature vector $H$. Assume
\begin{enumerate}
\item[i)] $|\phi|^2\leq 2|H|^2+2c-\frac{5c}{2}|T|^2$, and

\item[ii)] \begin{enumerate}
\item[a)]$|T|=0$, or

\item[b)] $|T|^2>\frac{2}{3}$ and $|H|^2\geq\frac{c|T|^2(1-|T|^2)}{3|T|^2-2}$.
\end{enumerate}
\end{enumerate}
Then either
\begin{enumerate}
\item $|\phi|^2=0$ and $\Sigma^2$ is a round sphere in $M^3(c)$, or

\item $|\phi|^2=2|H|^2+2c$ and $\Sigma^2$ is a torus $\mathbb{S}^1(r)\times\mathbb{S}^1(\sqrt{\frac{1}{c}-r^2})$, $r^2\neq \frac{1}{2c}$, in $M^3(c)$.
\end{enumerate}
\end{theorem2}

\section{Preliminaries}

Let $M^n(c)$ be a simply-connected $n$-dimensional manifold, with
constant sectional curvature $c$, and consider the product manifold
$\bar M=M^n(c)\times\mathbb{R}$. The expression of the curvature
tensor $\bar R$ of such a manifold can be obtained from
$$
\langle\bar R(X,Y)Z,W\rangle=c\{\langle d\pi Y, d\pi Z\rangle\langle d\pi X, d\pi W\rangle-\langle d\pi X, d\pi Z\rangle\langle d\pi Y, d\pi W\rangle\},
$$
where $\pi:\bar M=M^n(c)\times\mathbb{R}\rightarrow M^n(c)$ is the projection map. After a straightforward computation we get
\begin{equation}\label{eq:barR}
\begin{array}{ll}
\bar R(X,Y)Z=&c\{\langle Y, Z\rangle X-\langle X, Z\rangle Y-\langle Y,\xi\rangle\langle Z,\xi\rangle X+\langle X,\xi\rangle\langle Z,\xi\rangle Y\\ \\&+\langle X,Z\rangle\langle Y,\xi\rangle\xi-\langle Y,Z\rangle\langle X,\xi\rangle\xi\},
\end{array}
\end{equation}
where $\xi$ is the unit vector tangent to $\mathbb{R}$.

Now, let $\Sigma^2$ be an immersed surface in $\bar M$, and denote by $R$ its curvature tensor. Then, from the equation of Gauss
$$
\begin{array}{ll}
\langle R(X,Y)Z,W\rangle=&\langle\bar R(X,Y)Z,W\rangle\\ \\&+\sum_{\alpha=3}^{n+1}\{\langle A_{\alpha}Y,Z\rangle\langle A_{\alpha}X,W\rangle-\langle A_{\alpha}X,Z\rangle\langle A_{\alpha}Y,W\rangle\},
\end{array}
$$
we obtain
\begin{equation}\label{eq:R}
\begin{array}{ll}
R(X,Y)Z=&c\{\langle Y, Z\rangle X-\langle X, Z\rangle Y-\langle Y,T\rangle\langle Z,T\rangle X+\langle X,T\rangle\langle Z,T\rangle Y\\ \\&+\langle X,Z\rangle\langle Y,T\rangle T-\langle Y,Z\rangle\langle X,T\rangle T\}\\ \\&+\sum_{\alpha=3}^{n+1}\{\langle A_{\alpha}Y,Z\rangle A_{\alpha}X-\langle A_{\alpha}X,Z\rangle A_{\alpha}Y\},
\end{array}
\end{equation}
where $T$ is the component of $\xi$ tangent to the surface, $A$ is the shape operator defined by the equation of Weingarten
$$
\bar\nabla_XV=-A_VX+\nabla^{\perp}_XV,
$$
for any vector field $X$ tangent to $\Sigma^2$ and any vector
field $V$ normal to the surface.  Here $\bar\nabla$ is the
Levi-Civita connection on $\bar M$ and $\nabla^{\perp}$ is the connection in the normal bundle, and $A_{\alpha}=A_{E_{\alpha}}$, $\{E_{\alpha}\}_{\alpha=3}^{n+1}$ being a local orthonormal frame field in the normal bundle.

\begin{definition} If the mean curvature vector $H$ of the surface $\Sigma^2$ is
parallel in the normal bundle, i.e.
$\nabla^{\perp}H=0$, then $\Sigma^2$ is called a \textit{pmc surface}.
\end{definition}

We will end this section by recalling the Omori-Yau Maximum Principle, which shall be used later.

\begin{theorem}[\cite{Y2}]\label{OY} If $M$ is a complete Riemannian manifold with Ricci curvature bounded from below, then for any smooth function $u\in C^2(M)$ with $\sup_M u<+\infty$ there exists a sequence of points $\{p_k\}_{k\in\mathbb{N}}\subset M$ satisfying
$$
\lim_{k\rightarrow\infty}u(p_k)=\sup_Mu,\quad |\nabla u|(p_k)<\frac{1}{k}\quad\textnormal{and}\quad \Delta u(p_k)<\frac{1}{k}.
$$

\end{theorem}

\section{A formula for pmc surfaces in $M^n(c)\times\mathbb{R}$}\label{s1}

Let $\Sigma^2$ be an immersed surface in $M^n(c)\times\mathbb{R}$, with mean curvature vector $H$. In this section we shall prove a formula for the Laplacian of the squared norm of $A_V$, where $V$ is a normal vector to the surface, such that $V$ is parallel in the normal bundle, i.e. $\nabla^{\perp}V=0$, and $\trace A_V=\cst$.

\begin{lemma}\label{lemma:com}
If $U$ and $V$ are normal vectors to the surface and $V$ is parallel in the normal bundle, then $[A_V,A_U]=0$, i.e. $A_V$ commutes with $A_U$.
\end{lemma}

\begin{proof}
The conclusion follows easily, from the Ricci equation,
$$
\langle R^{\perp}(X,Y)V,U\rangle=\langle[A_V,A_U]X,Y\rangle+\langle\bar R(X,Y)V,U\rangle,
$$
since $R^{\perp}(X,Y)V=0$ and \eqref{eq:barR} implies that $\langle\bar R(X,Y)V,U\rangle=0$,
\end{proof}

Now, from the Codazzi equation,
$$
\begin{array}{cl}
\langle \bar R(X,Y)Z,V\rangle=&\langle\nabla^{\perp}_X\sigma(Y,Z),V\rangle-\langle\sigma(\nabla_XY,Z),V\rangle
-\langle\sigma(Y,\nabla_XZ),V\rangle\\ \\&-\langle\nabla^{\perp}_Y\sigma(X,Z),V\rangle+\langle\sigma(\nabla_YX,Z),V\rangle
+\langle\sigma(X,\nabla_YZ),V\rangle,
\end{array}
$$
where $\sigma$ is the second fundamental form of $\Sigma^2$, we get
$$
\begin{array}{rl}
\langle \bar R(X,Y)Z,V\rangle=&X(\langle A_VY,Z\rangle)-\langle\sigma(Y,Z),\nabla^{\perp}_XV\rangle
-\langle A_V(\nabla_XY),Z\rangle\\ \\&-\langle A_VY,\nabla_XZ\rangle-Y(\langle A_VX,Z\rangle)+\langle\sigma(X,Z),\nabla^{\perp}_YV\rangle
\\ \\&+\langle A_V(\nabla_YX),Z\rangle+\langle A_VX,\nabla_YZ\rangle\\ \\=&\langle(\nabla_XA_V)Y-(\nabla_YA_V)X,Z\rangle,
\end{array}
$$
since $\nabla^{\perp}V=0$. Therefore, using \eqref{eq:barR}, we obtain
\begin{equation}\label{eq:Codazzi}
(\nabla_XA_V)Y=(\nabla_YA_V)X+c\langle V,N\rangle(\langle Y,T\rangle X-\langle X,T\rangle Y),
\end{equation}
where $N$ is the normal part of $\xi$.

Next, we have
\begin{equation}\label{eq:Laplacian}
\frac{1}{2}\Delta|A_V|^2=|\nabla A_V|^2+\langle\nabla^2A_V,A_V\rangle,
\end{equation}
where we extended the metric $\langle,\rangle$ to the tensor space in the standard way.

In order to calculate the second term in the right hand side of \eqref{eq:Laplacian}, we shall use a method introduced in \cite{NS}.

Let us write
$$
C(X,Y)=\nabla_X(\nabla_YA_V)-\nabla_{\nabla_XY}A_V,
$$
and note that the fact that the torsion of $\nabla$ vanishes, together with the definition of the curvature tensor $R$ on the surface, implies
\begin{equation}\label{eq:C}
C(X,Y)=C(Y,X)+[R(X,Y),A_V].
\end{equation}
Next, consider an orthonormal basis $\{e_1,e_2\}$ in $T_p\Sigma^2$, $p\in\Sigma^2$, extend $e_1$ and $e_2$ to vector fields $E_1$ and $E_2$ in a neighborhood of $p$ such that $\nabla E_i=0$ at $p$, and let $X$ be a tangent vector field such that $\nabla X=0$. Obviously, we have, at $p$,
$$
(\nabla^2A_V)X=\sum_{i=1}^2\nabla_{E_i}(\nabla_{E_i}A_V)X=\sum_{i=1}^2C(E_i,E_i)X.
$$

Using equation \eqref{eq:Codazzi}, we get, at $p$,
$$
\begin{array}{lcl}
C(E_i,X)E_i&=&(\nabla_{E_i}(\nabla_XA_V))E_i-(\nabla_{\nabla_{E_i}X}A_V)E_i\\ \\&=&\nabla_{E_i}((\nabla_XA_V)E_i)-(\nabla_XA_V)(\nabla_{E_i}E_i)\\ \\&=&\nabla_{E_i}((\nabla_{E_i}A_V)X)+c\nabla_{E_i}(\langle V,N\rangle(\langle E_i,T\rangle X-\langle X,T\rangle E_i))
\end{array}
$$
and then
\begin{equation}\label{eq:1}
\begin{array}{lcl}
C(E_i,X)E_i&=&(\nabla_{E_i}(\nabla_{E_i}A_V))X+(\nabla_{E_i}A_V)(\nabla_{E_i}X)\\ \\&&-c\langle A_VE_i,T\rangle(\langle E_i,T\rangle X-\langle X,T\rangle E_i)\\ \\&&+c\langle V,N\rangle(\langle A_NE_i,E_i\rangle X-\langle A_NX,E_i\rangle E_i)\\ \\&=&C(E_i,E_i)X-c\langle A_VE_i,T\rangle(\langle E_i,T\rangle X-\langle X,T\rangle E_i)\\ \\&&+c\langle V,N\rangle(\langle A_NE_i,E_i\rangle X-\langle A_NX,E_i\rangle E_i),
\end{array}
\end{equation}
where we used $\sigma(E_i,T)=-\nabla^{\perp}_{E_i}N$ and $\nabla_{E_i}T=A_NE_i$, which follow from the fact that $\xi$ is parallel.

We also have
\begin{equation}\label{eq:2}
C(X,E_i)E_i=\nabla_{X}((\nabla_{E_i}A_V)E_i),
\end{equation}
and, from \eqref{eq:C}, \eqref{eq:1} and \eqref{eq:2}, we get
$$
\begin{array}{lcl}
C(E_i,E_i)X&=&C(E_i,X)E_i+c\langle A_VE_i,T\rangle(\langle E_i,T\rangle X-\langle X,T\rangle E_i)\\ \\&&-c\langle V,N\rangle(\langle A_NE_i,E_i\rangle X-\langle A_NX,E_i\rangle E_i)\\ \\&=&C(X,E_i)E_i+[R(E_i,X),A_V]E_i\\ \\&&+c\langle A_VE_i,T\rangle(\langle E_i,T\rangle X-\langle X,T\rangle E_i)\\ \\&&-c\langle V,N\rangle(\langle A_NE_i,E_i\rangle X-\langle A_NX,E_i\rangle E_i)
\end{array}
$$
which means that
$$
\begin{array}{lcl}
C(E_i,E_i)X&=&\nabla_{X}((\nabla_{E_i}A_V)E_i)+[R(E_i,X),A_V]E_i\\ \\&&+c\langle A_VE_i,T\rangle(\langle E_i,T\rangle X-\langle X,T\rangle E_i)\\ \\&&-c\langle V,N\rangle(\langle A_NE_i,E_i\rangle X-\langle A_NX,E_i\rangle E_i).
\end{array}
$$
As $A_V$ is symmetric, it follows that also $\nabla_{E_i}A_V$ is symmetric, and then, from \eqref{eq:Codazzi}, one obtains
$$
\begin{array}{lcl}
\langle\sum_{i=1}^2(\nabla_{E_i}A_V)E_i,Z\rangle&=&\sum_{i=1}^2\langle E_i,(\nabla_{E_i}A_V)Z\rangle=\sum_{i=1}^2\langle E_i,(\nabla_{Z}A_V)E_i\rangle\\ \\&&+c\langle V,N\rangle\sum_{i=1}^2\langle E_i,\langle Z,T\rangle E_i-\langle E_i,T\rangle Z\rangle\\ \\&=&\trace(\nabla_ZA_V)+c\langle V,N\rangle\langle T,Z\rangle\\ \\&=&Z(\trace A_V)+c\langle V,N\rangle\langle T,Z\rangle\\ \\&=&c\langle V,N\rangle\langle T,Z\rangle,
\end{array}
$$
for any vector $Z$ tangent to $\Sigma^2$, since $\trace A_V=\cst$.

Therefore, at $p$, we have
$$
\begin{array}{lcl}
(\nabla^2A_V)X&=&\sum_{i=1}^2C(E_i,E_i)X\\ \\&=&c\nabla_X(\langle V,N\rangle T)+\sum_{i=1}^2[R(E_i,X),A_V]E_i\\ \\&&+c\sum_{i=1}^2\langle A_VE_i,T\rangle(\langle E_i,T\rangle X-\langle X,T\rangle E_i)\\ \\&&-c\sum_{i=1}^2\langle V,N\rangle(\langle A_NE_i,E_i\rangle X-\langle A_NX,E_i\rangle E_i)
\end{array}
$$
and then, since $\bar\nabla_X\xi=0$ implies that $\sigma(X,T)=-\nabla^{\perp}_{X}N$ and $\nabla_{X}T=A_NX$,
\begin{equation}\label{eq:nabla1}
\begin{array}{lcl}
(\nabla^2A_V)X&=&\sum_{i=1}^2[R(E_i,X),A_V]E_i+c\{2\langle V,N\rangle A_NX-\langle A_VX,T\rangle T\\ \\&&+\langle A_VT,T\rangle X-\langle X,T\rangle A_VT-2\langle H,N\rangle\langle V,N\rangle X\}.
\end{array}
\end{equation}

From the Gauss equation \eqref{eq:R} of the surface $\Sigma^2$ and Lemma \ref{lemma:com}, we get
$$
\begin{array}{lcl}
\sum_{i=1}^2R(E_i,X)A_VE_i&=&c\sum_{i=1}^2\{\langle X,A_VE_i\rangle E_i-\langle E_i,A_VE_i\rangle X\\ \\&&-
\langle X,T\rangle\langle A_VE_i,T\rangle E_i+\langle E_i,T\rangle\langle A_VE_i,T\rangle X
\\ \\&&+\langle E_i,A_VE_i\rangle\langle X,T\rangle T-\langle X,A_VE_i\rangle\langle E_i,T\rangle T\}
\\ \\&&+\sum_{i=1}^2\sum_{\alpha=3}^{n+1}\{\langle A_{\alpha}X,A_VE_i\rangle A_{\alpha}E_i\\ \\&&-\langle A_{\alpha}E_i,A_VE_i\rangle A_{\alpha}X\},
\end{array}
$$
which means that
$$
\begin{array}{lcl}
\sum_{i=1}^2R(E_i,X)A_VE_i&=&c\{A_VX-(\trace A_V)X+(\trace A_V)\langle X,T\rangle
T\\ \\&&-\langle A_VX,T\rangle T
-\langle X,T\rangle A_VT+\langle A_VT,T\rangle X\}
\\ \\&&+\sum_{\alpha=3}^{n+1}\{A_VA_{\alpha}^2 X-(\trace(A_VA_{\alpha}))A_{\alpha}X\},
\end{array}
$$
and
$$
\begin{array}{lcl}
\sum_{i=1}^2A_VR(E_i,X)E_i&=&c\sum_{i=1}^2\{\langle X,E_i\rangle A_VE_i-\langle E_i,E_i\rangle A_VX\\ \\&&-\langle X,T\rangle\langle E_i,T\rangle A_VE_i+\langle E_i,T\rangle\langle E_i,T\rangle A_VX\\ \\&&+\langle E_i,E_i\rangle\langle X,T\rangle A_VT-\langle X,E_i\rangle\langle E_i,T\rangle A_VT\}\\ \\&&+\sum_{i=1}^2\sum_{\alpha=3}^{n+1}\{\langle A_{\alpha}X,E_i\rangle A_VA_{\alpha}E_i\\ \\&&-\langle A_{\alpha}E_i,E_i\rangle A_VA_{\alpha}X\}\\ \\&=&-c(1-|T|^2)A_VX\\ \\&&+\sum_{\alpha=3}^{n+1}\{A_VA_{\alpha}^2X-(\trace A_{\alpha})A_VA_{\alpha}X\}.
\end{array}
$$
Finally, replacing in equation \eqref{eq:nabla1}, we find
$$
\begin{array}{ll}
(\nabla^2A_V)X=&c\{(2-|T|^2)A_VX+2\langle A_VT,T\rangle X-2\langle A_VX,T\rangle T-2\langle X,T\rangle A_VT\\ \\&+2\langle V,N\rangle A_NX-2\langle H,N\rangle\langle V,N\rangle X\\ \\&-(\trace A_V)X+(\trace A_V)\langle X,T\rangle T\}\\ \\&+\sum_{\alpha=3}^{n+1}\{(\trace A_{\alpha})A_VA_{\alpha}X-(\trace(A_VA_{\alpha}))A_{\alpha}X\},
\end{array}
$$
and, after a straightforward computation,
$$
\begin{array}{lcl}
\langle\nabla^2A_V,A_V\rangle&=&\sum_{i=1}^2\langle(\nabla^2A_V)E_i,A_VE_i\rangle\\ \\&=&c\{(2-|T|^2)|A_V|^2-4|A_VT|^2+3(\trace A_V)\langle A_VT,T\rangle\\ \\&&+2(\trace(A_NA_V))\langle V,N\rangle-(\trace A_V)^2\\ \\&&-2(\trace A_V)\langle H,N\rangle\langle V,N\rangle\}\\ \\&&+\sum_{\alpha=3}^{n+1}\{(\trace A_{\alpha})(\trace(A_V^2A_{\alpha}))-(\trace(A_VA_{\alpha}))^2\}.
\end{array}
$$

Thus, from \eqref{eq:Laplacian}, we obtain the following

\begin{proposition}\label{p:delta} Let $\Sigma^2$ be an immersed surface in $M^n(c)\times\mathbb{R}$. If $V$ is a normal vector field, parallel in the normal bundle, with $\trace A_V=\cst$, then
\begin{equation}\label{eq:delta2}
\begin{array}{lcl}
\frac{1}{2}\Delta|A_V|^2&=&|\nabla A_V|^2+c\{(2-|T|^2)|A_V|^2-4|A_VT|^2+3(\trace A_V)\langle A_VT,T\rangle\\ \\&&+2(\trace(A_NA_V))\langle V,N\rangle-(\trace A_V)^2\\ \\&&-2(\trace A_V)\langle H,N\rangle\langle V,N\rangle\}\\ \\&&+\sum_{\alpha=3}^{n+1}\{(\trace A_{\alpha})(\trace(A_V^2A_{\alpha}))-(\trace(A_VA_{\alpha}))^2\},
\end{array}
\end{equation}
where $\{E_{\alpha}\}_{\alpha=3}^{n+1}$ is a local orthonormal frame field in the normal bundle.
\end{proposition}

\begin{corollary}\label{coro:phi3} If $\Sigma^2$ is an immersed non-minimal pmc surface in $M^n(c)\times\mathbb{R}$ and $\phi_H$ is the operator defined by $\phi_H=\frac{1}{|H|}A_H-|H|\id$, then
\begin{equation}\label{eq:deltaH}
\begin{array}{lcl}
\frac{1}{2}\Delta|\phi_H|^2&=&|\nabla \phi_H|^2+\{c(2-3|T|^2)+4|H|^2-|\sigma|^2\}|\phi_H|^2\\ \\&&-2c|H|\langle\phi_HT,T\rangle+\frac{2c}{|H|}\langle H,N\rangle\trace(A_N\phi_H).
\end{array}
\end{equation}
\end{corollary}

\begin{proof} From the definition of $\phi_H$ we have $\nabla\phi_H=\frac{1}{|H|}\nabla A_H$, $|\phi_H|^2=\frac{1}{|H|^2}|A_H|^2-2|H|^2$ and
$
\frac{1}{|H|^2}|A_HT|^2=\frac{1}{2}|T|^2|\phi_H|^2+|H|^2|T|^2+2|H|\langle\phi_HT,T\rangle$,
where we used the fact that $|\phi_HT|^2=\frac{1}{2}|T|^2|\phi_H|^2$, which can be easily verified by working in a basis that diagonalizes $\phi_H$ and taking into account that $\trace \phi_H=0$.

Next, from equation \eqref{eq:nabla1}, with $V=H$, we get $\langle\nabla^2A_H,\id\rangle=0$ and, therefore, from Proposition~\ref{p:delta}, it follows
\begin{equation}\label{eq:phipartial}
\begin{array}{lcl}
\frac{1}{2}\Delta|\phi_H|^2&=&|\nabla \phi_H|^2+c(2-3|T|^2)|\phi_H|^2-2c|H|\langle\phi_HT,T\rangle\\ \\&&+\frac{2c}{|H|}\langle H,N\rangle\trace(A_N\phi_H)\\ \\&&+\sum_{\alpha=3}^{n+1}\{(\trace A_{\alpha})(\trace(\phi_H^2A_{\alpha}))-(\trace(\phi_HA_{\alpha}))^2\}.
\end{array}
\end{equation}

Now, let us consider the local orthonormal frame field
$\{E_3=\frac{H}{|H|},E_4,\ldots,E_{n+1}\}$ in the normal bundle. One sees that $\trace A_3=2|H|$, $\trace A_{\alpha}=0$,
$\alpha>3$, and hence
$$
\sum_{\alpha=3}^{n+1}(\trace A_{\alpha})(\trace(\phi_H^2A_{\alpha}))=2|H|(\trace(\phi_H^2A_3))=2(\trace(\phi_H^2A_H)).
$$
From the definition of $\phi_H$, we get $
\phi_H^2A_H=|H|\phi_H^3+|H|^2\phi_H^2$
and, since $\trace\phi_H^3=0$,
$$
2(\trace\phi_H^2A_H)=2|H|\trace\phi_H^3+2|H|^2\trace \phi_H^2=2|H|^2|\phi_H|^2.
$$
We have just proved that
\begin{equation}\label{eq:tr1}
\sum_{\alpha=3}^{n+1}(\trace A_{\alpha})(\trace(\phi_H^2A_{\alpha}))=2|H|^2|\phi_H|^2.
\end{equation}

As we have seen in Lemma \ref{lemma:com}, since $H$ is parallel, we have that $A_H$ commutes with $A_U$, for any normal vector field
$U$. Then, either there exists a basis that diagonalizes $A_U$, for
all vectors $U$ normal to $\Sigma^2$, or the surface is
pseudo-umbilical, i.e. $A_{H}=|H|^2\id$. Moreover, since the map
$p\in\Sigma^2\rightarrow(A_H-\mu\id)(p)$, where $\mu$ is a constant,
is analytic, it follows that if $H$ is an umbilical direction, then
this either holds on the whole surface or only on a closed set
without interior points.

Since $H$ is an umbilical direction everywhere implies
that $\phi_H$ vanishes on the surface, and then \eqref{eq:deltaH} is
verified, we will study only the case when $H$ is an umbilical
direction on a closed set without interior points, which means that
$H$ is not umbilical in an open dense set. We will work on this set
and then will extend our result throughout $\Sigma^2$ by continuity.

Let $\{e_1,e_2\}$ be a basis that diagonalizes $A_U$, for all
vectors $U$ normal to the surface. Therefore, with respect to this
basis, for $\alpha>3$, since $\trace A_{\alpha}=\trace\phi_H=0$, we
have $ A_{\alpha}=\left(\begin{array}{cc}\mu_{\alpha}&0\\
\\0&-\mu_{\alpha}\end{array}\right)$ and
$\phi_H=\left(\begin{array}{cc}a&0\\ \\0&-a\end{array}\right)$, and
then
\begin{equation}\label{eq:alfa}
(\trace(\phi_HA_{\alpha}))^2=4a^2\mu_{\alpha}^2=|A_{\alpha}|^2|\phi_H|^2.
\end{equation}
We also obtain $\phi_HA_3=\phi_H^2+|H|\phi_H$, which leads to
\begin{equation}\label{eq:a3}
(\trace(\phi_HA_3))^2=|\phi_H|^4=(|A_3|^2-2|H|^2)|\phi_H|^2.
\end{equation}
Finally, replacing \eqref{eq:tr1}, \eqref{eq:alfa} and \eqref{eq:a3} in \eqref{eq:phipartial}, we get equation
\eqref{eq:deltaH}.
\end{proof}

\section{A Simons type equation and applications}

In the following we shall derive a Simons type equation for
non-minimal pmc surfaces in $M^3(c)\times\mathbb{R}$ and then we
will use it in order to characterize some of these surfaces.

Everywhere in this section $\Sigma^2$ will be an immersed
non-minimal pmc surface in a product space $M^3(c)\times\mathbb{R}$,
with mean curvature vector $H$ and the Gaussian curvature $K$.

Let us consider the local orthonormal frame field $\{E_3=\frac{H}{|H|},E_4\}$ in
the normal bundle and denote by $\phi_3=\phi_H=A_3-|H|\id$ and $\phi_4=A_4$.
The normal part of $\xi$ can be written as $N=\nu_3E_3+\nu_4E_4$, where $\nu_3=\langle\xi, E_3\rangle$ and
$\nu_4=\langle\xi,E_4\rangle$.

Since $H$ is parallel in the normal bundle, it follows that also $E_4$ is parallel in the
normal bundle. We use the same argument as in the proof of Corollary
\ref{coro:phi3} to see that either $H$ is an umbilical direction on
the whole surface or it is not umbilical on an open dense set. In
both cases, it is easy to verify that
$$
(\trace A_3)(\trace(\phi_4^2A_3))=2|H|^2|\phi_4|^2\quad\textnormal{and}
\quad(\trace(\phi_4A_3))^2=(|A_3|^2-2|H|^2)|\phi_4|^2,
$$
and then, since $(\trace\phi_4^2)^2=|\phi_4|^4$, $\trace\phi_4=0$
and $|\phi_4T|^2=\frac{1}{2}|T|^2|\phi_4|^2$, from Proposition
\ref{p:delta}, we can derive the following formula for the Laplacian
of $|\phi_4|^2$:
\begin{equation}\label{eq:deltaphi4}
\frac{1}{2}\Delta|\phi_4|^2=|\nabla\phi_4|^2+\{c(2-3|T|^2)+4|H|^2-|\sigma|^2\}|\phi_4|^2+2c\nu_4\trace(A_N\phi_4).
\end{equation}

Next, let $\phi$ be the traceless part of the second fundamental
form $\sigma$ of the surface, given by
$$
\phi(X,Y)=\sigma(X,Y)-\langle X,Y\rangle H=\sum_{\alpha=3}^{4}\langle\phi_{\alpha}X,Y\rangle E_{\alpha}.
$$
We have $|\phi|^2=|\phi_3|^2+|\phi_4|^2=|\sigma|^2-2|H|^2$ and then, using \eqref{eq:deltaH} and \eqref{eq:deltaphi4}, we can state the following

\begin{proposition}\label{p:phi} If $\Sigma^2$ is an immersed non-minimal pmc surface in $M^3(c)\times\mathbb{R}$ and $\phi$ is the traceless part of its second fundamental form, then
\begin{equation}\label{eq:deltaphi}
\begin{array}{cl}
\frac{1}{2}\Delta|\phi|^2=&|\nabla\phi_3|^2+|\nabla\phi_4|^2-|\phi|^4+\{c(2-3|T|^2)+2|H|^2\}|\phi|^2\\ \\&-2c\langle\phi(T,T),H\rangle+2c|\nu_3\phi_3+\nu_4\phi_4|^2,
\end{array}
\end{equation}
or, equivalently,
$$
\begin{array}{cl}
\frac{1}{2}\Delta|\phi|^2=&|\nabla\phi_3|^2+|\nabla\phi_4|^2-|\phi|^4+\{c(2-3|T|^2)+2|H|^2\}|\phi|^2\\ \\&-2c\langle\phi(T,T),H\rangle+2c|A_N|^2-4c\langle H,N\rangle^2.
\end{array}
$$
\end{proposition}

\begin{theorem}\label{th:appl1} Let $\Sigma^2$ be an immersed complete non-minimal pmc surface in
$\bar M=M^3(c)\times\mathbb{R}$, with $c<0$, such that
$$
\sup_{\Sigma^2}|\phi|^2< 2|H|^2+c(4-5|T|^2)\quad\textnormal{and}\quad\langle\phi(T,T),H\rangle\geq 0.
$$
Then $|\phi|^2=0$ and $\Sigma^2$ is a round sphere in $M^3(c)$.
\end{theorem}

\begin{proof} First, using the Schwarz inequality, we get
$$
|\nu_3\phi_3+\nu_4\phi_4|^2\leq(\nu_3^2+\nu_4^2)(|\phi_3|^2+|\phi_4|^2)=(1-|T|^2)|\phi|^2,
$$
and this, together with \eqref{eq:deltaphi} and the hypothesis, leads to
\begin{equation}\label{ineq}
\begin{array}{rl}
\frac{1}{2}\Delta|\phi|^2&\geq-|\phi|^4+\{c(4-5|T|^2)+2|H|^2\}|\phi|^2-2c\langle\phi(T,T),H\rangle\\ \\&\geq \{-|\phi|^2+c(4-5|T|^2)+2|H|^2\}|\phi|^2\\ \\&\geq 0.
\end{array}
\end{equation}

On the other hand, we can prove that the Gaussian curvature $K$ of the surface is bounded from below. Indeed, from \eqref{eq:R}, we have
$$
2K=2c(1-|T|^2)+4|H|^2-|\sigma|^2=2c(1-|T|^2)+2|H|^2-|\phi|^2>c(3|T|^2-2)\geq c.
$$

Therefore, since $\Sigma^2$ is complete, the Omori-Yau Maximum Principle holds on the surface. We take $u=|\phi|^2$ in Theorem \ref{OY} and it follows that there exists a sequence of points $\{p_k\}_{k\in\mathbb{N}}\subset \Sigma^2$ such that
$$
\lim_{k\rightarrow\infty}|\phi|^2(p_k)=\sup_{\Sigma^2}|\phi|^2\quad\textnormal{and}\quad \Delta |\phi|^2(p_k)<\frac{1}{k}.
$$
From \eqref{ineq}, we get that
$\lim_{k\rightarrow\infty}|\phi|^2(p_k)=0$ and then
$\sup_{\Sigma^2}|\phi|^2=0$, which implies $|\phi|^2=0$. Since
$\phi_3=A_3-|H|\id=0$, it follows that $H$ is an umbilical direction
and this implies that $\Sigma^2$ is a totally umbilical surface in
$M^3(c)$ (see Lemma $3$ in \cite{AdCT}). Therefore, $\Sigma^2$ is a
horosphere or a round sphere. Since $|\phi|^2<2|H|^2+4c$ implies
that $K>-c>0$, $\Sigma^2$ cannot be flat, and then we conclude that
the surface is a round sphere.
\end{proof}

\begin{theorem}\label{th:appl2} Let $\Sigma^2$ be an immersed complete non-minimal pmc surface in
$\bar M=M^3(c)\times\mathbb{R}$, with $c>0$, such that
$$
|\phi|^2\leq 2|H|^2+c(2-3|T|^2)\quad\textnormal{and}\quad\langle\phi(T,T),H\rangle\leq 0.
$$
Then $\xi$ is normal to the surface and either
\begin{enumerate}
\item $|\phi|^2=0$ and $\Sigma^2$ is a round sphere in $M^3(c)$, or

\item $|\phi|^2=2|H|^2+2c$ and $\Sigma^2$ is a torus $\mathbb{S}^1(r)\times\mathbb{S}^1(\sqrt{\frac{1}{c}-r^2})$, $r^2\neq \frac{1}{2c}$, in $M^3(c)$.
\end{enumerate}
\end{theorem}

\begin{proof} From the Gauss equation \eqref{eq:R} of the surface,
since $|\phi|^2\leq 2|H|^2+c(2-3|T|^2)$, we obtain that
$$
2K=2c(1-|T|^2)+4|H|^2-|\sigma|^2=2c(1-|T|^2)+2|H|^2-|\phi|^2\geq c|T|^2\geq 0
$$
and then, a result of A.~Huber in \cite{H} implies that $\Sigma^2$ is a parabolic space.

On the other hand, from \eqref{eq:deltaphi}, we see that
$\Delta|\phi|^2\geq 0$ and then $|\phi|^2$ is a bounded subharmonic
function on a parabolic space. Therefore $|\phi|^2$ is a constant
and, again using \eqref{eq:deltaphi}, we get that
$$
\{-|\phi|^2+c(2-3|T|^2)+2|H|^2\}|\phi|^2=0,\quad\langle\phi(T,T),H\rangle=0
\quad\textnormal{and}\quad\nu_3\phi_3+\nu_4\phi_4=0.
$$
Now we can split our study in two cases.

\textbf{\underline{Case I: $|\phi|^2=0$.}} This case can be handled
exactly like in the proof of Theorem \ref{th:appl1}.

\textbf{\underline{Case II: $|\phi|^2=c(2-3|T|^2)+2|H|^2$.}} Since
$|\phi|^2$ is a constant, it follows that $|T|^2$ is a constant and
then that $\langle\nabla_XT,T\rangle=0$, for any vector $X$ tangent
to $\Sigma^2$. As $\bar\nabla_X\xi=0$ implies that $\nabla_XT=A_NX$,
we have $\langle A_NX,T\rangle=0$. But $\nu_3\phi_3+\nu_4\phi_4=0$
means that $A_N=\langle H,N\rangle\id$, which implies that $\langle
H,N\rangle\langle X,T\rangle=0$, for any tangent vector $X$.
Therefore, either $\xi$ is orthogonal to $\Sigma^2$ at any point, or
only on a closed set without interior points. In this second case,
$\langle H,\xi\rangle=0$ holds on an open dense set. In this set, we
also have $\langle\bar\nabla_TH,\xi\rangle=-\langle
A_HT,T\rangle=0$. We have just shown that $\langle A_3T,T\rangle=0$
on an open dense set and, since we also know that
$$
\langle\phi(T,T),H\rangle=|H|\langle\phi_3 T,T\rangle=|H|(\langle A_3T,T\rangle-|H||T|^2)=0,
$$
one gets that $T=0$ on the whole surface. Hence, $\Sigma^2$ lies in
$M^3(c)$ and its Gaussian curvature is $K=\frac{c}{2}|T|^2=0$. We
use a similar argument to that in the proof of Theorem $1.5$ in
\cite{AdC} (see also \cite{WS}) to conclude.
\end{proof}

\section{A gap Theorem for pmc surfaces with non-negative Gaussian curvature}

In this Section we will prove our main results, Theorem \ref{thm:spheres} and Theorem \ref{th:appl3}. In order to do
that, let us consider an immersed pmc surface $\Sigma^2$ in
$M^n(c)\times\mathbb{R}$, and first we shall compute the Laplacian of
$|T|^2$.

Let $\{e_1,e_2\}$ be an orthonormal in $T_p\Sigma^2$,
$p\in\Sigma^2$, extend $e_1$ and $e_2$ to vector fields $E_1$ and
$E_2$ in a neighborhood of $p$ such that $\nabla E_i=0$ at $p$. At
$p$, we have
$$
\begin{array}{lcl}
\frac{1}{2}\Delta|T|^2&=&\sum_{i=1}^2(\langle\nabla_{E_i}T,\nabla_{E_i}T\rangle+\langle\nabla_{E_i}\nabla_{E_i}T,T\rangle)\\ \\
&=&|A_N|^2+\sum_{i=1}^2\langle\nabla_{E_i}A_NE_i,T\rangle
\end{array}
$$
and, since $\nabla_XA_N$ is symmetric,
$$
\begin{array}{lcl}
\sum_{i=1}^2\langle\nabla_{E_i}A_NE_i,T\rangle&=&\sum_{i=1}^2\langle(\nabla_{E_i}A_N)E_i,T\rangle\\ \\
&=&\sum_{i=1}^2\langle(\nabla_{E_i}A_N)T,E_i\rangle=\sum_{i=1}^2\langle\nabla_{E_i}A_NT
-A_N\nabla_{E_i}T,E_i\rangle\\ \\&=&\sum_{i=1}^2\langle\nabla_{E_i}\nabla_TT
-\nabla_{\nabla_{E_i}T}T,E_i\rangle\\ \\&=&\sum_{i=1}^2\langle\nabla_{E_i}\nabla_TT
+\nabla_{[T,E_i]}T,E_i\rangle\\ \\&=&\sum_{i=1}^2(\langle\nabla_T\nabla_{E_i}T,E_i\rangle
-\langle R(T,E_i)T,E_i\rangle)\\ \\&=&|T|^2K+\sum_{i=1}^2\langle\nabla_TA_NE_i,E_i\rangle\\ \\&=&
|T|^2K+\sum_{i=1}^2T(\langle A_NE_i,E_i\rangle)=|T|^2K+T(\trace A_N)\\ \\&=&|T|^2K+2T(\langle H,N\rangle)
=|T|^2K-2\langle\sigma(T,T),H\rangle\\ \\&=&c|T|^2(1-|T|^2)-\frac{1}{2}|T|^2|\phi|^2-2\langle\phi(T,T),H\rangle-|T|^2|H|^2,
\end{array}
$$
where we used $\nabla_XT=A_NX$ and $\nabla^{\perp}_XN=-\sigma(X,T)$, and $\phi$ is the traceless part of the second fundamental form $\sigma$ of the surface.

Now, we can state

\begin{proposition}\label{pT} If $\Sigma^2$ is an immersed pmc
surface in $M^n(c)\times\mathbb{R}$, then
\begin{equation}\label{eq:deltaT}
\frac{1}{2}\Delta|T|^2=|A_N|^2-\frac{1}{2}|T|^2|\phi|^2-2\langle\phi(T,T),H\rangle+c|T|^2(1-|T|^2)-|T|^2|H|^2.
\end{equation}
\end{proposition}

In \cite{AdCT1} the authors introduced a holomorphic differential, defined on pmc surfaces in $M^n(c)\times\mathbb{R}$ (when $n=2$ this is just the Abresch-Rosenberg differential, defined in \cite{AR}). This holomorphic differential is the $(2,0)$-part of the quadratic form $Q$, given by
$$
Q(X,Y)=2\langle\sigma(X,Y),H\rangle-c\langle X,\xi\rangle\langle Y,\xi\rangle.
$$
Using this result and Proposition \ref{pT}, we can characterize pmc $2$-spheres $\Sigma^2$ immersed in $M^n(c)\times\mathbb{R}$, whose second fundamental form satisfies a certain condition.

\begin{theorem}\label{thm:spheres} Let $\Sigma^2$ be an immersed pmc $2$-sphere in $M^n(c)\times\mathbb{R}$, such that
\begin{enumerate}
\item $|T|^2=0$ or $|T|^2\geq\frac{2}{3}$ and $|\sigma|^2\leq c(2-3|T|^2)$, if $c<0$;

\item $|T|^2\leq\frac{2}{3}$ and $|\sigma|^2\leq c(2-3|T|^2)$, if $c>0$.
\end{enumerate}
Then, $\Sigma^2$ is either a minimal surface in a totally umbilical hypersurface of $M^n(c)$ or a standard sphere in $M^3(c)$.
\end{theorem}

\begin{proof} If $\xi$ is orthogonal to the surface in an open connected subset, then this subset lies in $M^n(c)$, and by analyticity, it follows that $\Sigma^2$ lies in $M^n(c)$. In this case we use Theorem $4$ in \cite{Y} to conclude.

Next, let us assume that we are not in the previous case. Then, we can choose an orthonormal frame $\{e_1,e_2\}$ on the surface, such that $e_1=\frac{T}{|T|}$. Since $\Sigma^2$ is a sphere and the $(2,0)$-part of $Q$ is holomorphic, it follows that it vanishes on the surface. This means that $Q(e_1,e_1)=Q(e_2,e_2)$ and $Q(e_1,e_2)=0$. From $Q(e_1,e_1)=Q(e_2,e_2)$, we get $2\langle\sigma(e_1,e_1)-\sigma(e_2,e_2),H\rangle=c|T|^2$. We have
$$
\langle\phi(T,T),H\rangle=\langle\sigma(T,T),H\rangle-|T|^2|H|^2=\frac{1}{2}|T|^2\langle\sigma(e_1,e_1)
-\sigma(e_2,e_2),H\rangle=\frac{1}{4}c|T|^4,
$$
and then \eqref{eq:deltaT} becomes
$$
\frac{1}{2}\Delta|T|^2=|A_N|^2+\frac{1}{2}|T|^2(-|\sigma|^2+c(2-3|T|^2))\geq 0.
$$
Since $\Sigma^2$ is a sphere and $|T|^2$ is a bounded subharmonic function, it results that $|T|^2$ is constant, and then that $|A_N|^2=0$ and $|T|^2(-|\sigma|^2+c(2-3|T|^2))=0$. Since $A_N=0$ and $\xi$ is parallel, we obtain that $\nabla_XT=0$, for any tangent vector $X$, which implies that $K=0$. Since our surface is a sphere, this is a contradiction. Therefore, $\Sigma^2$ lies in $M^n(c)$ and, again using Theorem $4$ in \cite{Y} (see also Theorem $2$ in \cite{AdCT}) we come to the conclusion.
\end{proof}

Next, assume that $\Sigma^2$ is an immersed non-minimal pmc surface
in $M^3(c)\times\mathbb{R}$. Then, from \eqref{eq:deltaphi} and
\eqref{eq:deltaT}, we obtain
\begin{equation}\label{eq:deltasum}
\begin{array}{lcl}
\frac{1}{2}\Delta(|\phi|^2-c|T|^2)&=&|\nabla\phi_3|^2+|\nabla\phi_4|^2+\{-|\phi|^2+\frac{c}{2}(4-5|T|^2)+2|H|^2\}|\phi|^2
\\ \\&&+c|A_N|^2-4c\langle H,N\rangle^2+c|T|^2|H|^2-c^2|T|^2(1-|T|^2).
\end{array}
\end{equation}

Using this equation, we shall prove the following

\begin{theorem}\label{th:appl3} Let $\Sigma^2$ be an immersed
complete non-minimal pmc surface in $\bar M=M^3(c)\times\mathbb{R}$,
with $c>0$. Assume
\begin{enumerate}
\item[i)] $|\phi|^2\leq 2|H|^2+2c-\frac{5c}{2}|T|^2$, and

\item[ii)] \begin{enumerate}
\item[a)]$|T|=0$, or

\item[b)] $|T|^2>\frac{2}{3}$ and $|H|^2\geq\frac{c|T|^2(1-|T|^2)}{3|T|^2-2}$.
\end{enumerate}
\end{enumerate}
Then either
\begin{enumerate}
\item $|\phi|^2=0$ and $\Sigma^2$ is a round sphere in $M^3(c)$, or

\item $|\phi|^2=2|H|^2+2c$ and $\Sigma^2$ is a torus $\mathbb{S}^1(r)\times\mathbb{S}^1(\sqrt{\frac{1}{c}-r^2})$, $r^2\neq \frac{1}{2c}$, in $M^3(c)$.
\end{enumerate}
\end{theorem}

\begin{proof} If $|T|^2=0$, it is easy to see that
\eqref{eq:deltasum} implies
$$
\frac{1}{2}\Delta(|\phi|^2-c|T|^2)\geq\{-|\phi|^2+2c+2|H|^2\}|\phi|^2\geq 0.
$$

Next, let us assume that
$|T|^2>\frac{2}{3}$ and $|H|^2\geq\frac{c|T|^2(1-|T|^2)}{3|T|^2-2}$. Since
$$
|A_N|^2-2\langle H,N\rangle=|\nu_3\phi_3+\nu_4\phi_4|^2\geq 0
$$
and,
from the Schwarz inequality, we have
\begin{equation}\label{eq:schwarz}
\langle H,N\rangle^2\leq
|N|^2|H|^2=(1-|T|^2)|H|^2,
\end{equation}
then, from \eqref{eq:deltasum}, it follows that
$$
\begin{array}{lcl}
\frac{1}{2}\Delta(|\phi|^2-c|T|^2)&\geq&\{-|\phi|^2+\frac{c}{2}(4-5|T|^2)+2|H|^2\}|\phi|^2
\\ \\&&+c(3|T|^2-2)|H|^2-c^2|T|^2(1-|T|^2)\\ \\&\geq&0.
\end{array}
$$

The Gaussian curvature of the surface satisfies
$$
2K=2c(1-|T|^2)+2|H|^2-|\phi|^2\geq\frac{1}{2}c|T|^2\geq 0,
$$
which means that $\Sigma^2$ is a parabolic space. Now, since
$|\phi|^2-c|T|^2$ is a bounded subharmonic function, it follows that
it is constant. Therefore, either $|\phi|^2=0$ or
$|\phi|^2=2|H|^2+2c-\frac{5c}{2}|T|^2$. The first case can be
handled exactly as in the proof of Theorem \ref{th:appl1}. As for
the second case, from \eqref{eq:deltasum}, we have that
$\nu_3\phi_3+\nu_4\phi_4=0$, which means that $A_N=\langle
H,N\rangle\id$, and the equality holds in \eqref{eq:schwarz}, i.e. either
$N=\nu_3H$ or $\xi$ is tangent to the surface. If $\xi$ is tangent to the surface, from \eqref{eq:deltasum} and the hypothesis, it follows that $\Sigma^2$ is a minimal surface, which is a contradiction. Hence, $N=\nu_3H$ and we get $A_H=|H|^2\id$, which, again using
Lemma $3.1$ in \cite{AdCT}, implies that the surface lies in
$M^3(c)$. We then come to the conclusion in the same way as in the proof
of the second part of Theorem \ref{th:appl2}.
\end{proof}


\begin{thebibliography}{99}

\bibitem{AR} U. Abresch and H. Rosenberg, \textit{A Hopf differential for constant mean curvature
surfaces in $\mathbb{S}^2\times\mathbb{R}$ and
$\mathbb{H}^2\times\mathbb{R}$}, Acta Math. 193(2004), 141-–174.

\bibitem{AR2} U. Abresch and H. Rosenberg, {\it Generalized Hopf differentials}, Mat. Contemp.
28(2005), 1--28.

\bibitem{AdC} H.~Alencar and M. do Carmo, {\it Hypersurfaces with constant mean curvature in spheres},  Proc. Amer. Math. Soc.  120(1994), 1223--1229.

\bibitem{AdCT1} H.~Alencar, M. do Carmo and R.~Tribuzy, {\it A theorem of Hopf and the Cauchy-Riemann inequality}, Comm. Anal. Geom. 15(2007), 283--298.

\bibitem{AdCT} H.~Alencar, M.~do Carmo and R.~Tribuzy, \textit{A
Hopf Theorem for ambient spaces of dimensions higher than three},
J. Differential Geometry 84(2010), 1--17.

\bibitem{B} M.~Batista, \textit{Simons type equation in $\mathbb{S}^2\times\mathbb{R}$ and
$\mathbb{H}^2\times\mathbb{R}$ and applications}, preprint, 2010.

\bibitem{CL} B.-Y.~Chen and G. D.~Ludden, \textit{Surfaces with mean curvature vector parallel
in the normal bundle}, Nagoya Math. J. 47(1972), 161--167.

\bibitem{C} S.-S. Chern, \textit{On surfaces of constant mean curvature in a three-dimensional space of constant curvature}, Geometric dynamics (Rio de Janeiro, 1981), Lecture Notes in Math. 1007, Springer, Berlin, 1983, 104--108.

\bibitem{E} J.~Erbacher, \textit{Isometric immersions of constant mean curvature and triviality of the normal connection}, Nagoya Math. J. 45(1971), 139--165.

\bibitem{DF} D.~Ferus, \textit{The torsion form of submanifolds in $E^N$}, Math. Ann. 193(1971), 114-–120.

\bibitem{DH} D. A.~Hoffman, \textit{Surfaces in constant curvature manifolds with parallel mean curvature vector field}, Bull. Amer. Math. Soc. 78(1972), 247--250.

\bibitem{HH} H.~Hopf, \textit{Differential Geometry in the Large},
Lecture Notes in Math. 1000, Springer-Verlag, 1983.

\bibitem{H} A.~Huber, \textit{On subharmonic functions and differential geometry in the
large}, Comm. Math. Helv. 32(1957), 13--71.

\bibitem{NS} K.~Nomizu and B.~Smyth, \textit{A formula of Simons' type and hypersurfaces with constant mean
curvature}, J. Differential Geometry 3(1969), 367--377.

\bibitem{WS} W.~Santos, {\it Submanifolds with parallel mean curvature vector in spheres}, T\^ohoku Math. J. 46(1994), 403--415.

\bibitem{JS} J.~Simons, \textit{Minimal varieties in Riemannian
manifolds}, Ann. of Math. 88(1968), 62--105.

\bibitem{S} B.~Smyth, \textit{Submanifolds of constant mean curvature}, Math. Ann 205(1973), 265--280.

\bibitem{Y} S.-T.~Yau, \textit{Submanifolds with constant mean
curvature}, Amer. J. Math. 96(1974), 346--366.

\bibitem{Y2} S.-T.~Yau, \textit{Harmonic functions on complete Riemannian manifolds}, Commun. Pure. Appl. Math. 28(1975), 201--228.

\end{thebibliography}
\end{document}